\documentclass[12pt]{amsart}
\usepackage{a4wide}
\usepackage{amssymb}
\usepackage{amsthm}
\usepackage{amsmath}
\usepackage{amscd}
\usepackage{verbatim}
\usepackage[all]{xy}
\usepackage{longtable, lscape}

\theoremstyle{plain}
\newtheorem{theorem}{Theorem}[section]

\newtheorem{lemma}[theorem]{Lemma}
\newtheorem{proposition}[theorem]{Proposition}

\newtheorem{conjecture}[theorem]{Conjecture}
\theoremstyle{definition}
\newtheorem{definition}[theorem]{Definition}

\theoremstyle{remark}
\newtheorem{remark}[theorem]{Remark}

\newcommand{\Q}{\mathbb{Q}}
\newcommand{\Z}{\mathbb{Z}}

\newcommand{\C}{\mathbb{C}}

\renewcommand{\H}{\mathbb{H}}

\newcommand{\zxz}[4]{\begin{pmatrix} #1 & #2 \\ #3 & #4 \end{pmatrix}}
\newcommand{\abcd}{\zxz{a}{b}{c}{d}}

\newcommand{\kzxz}[4]{\left(\begin{smallmatrix} #1 & #2 \\ #3 & #4\end{smallmatrix}\right) }

\newcommand{\calD}{\mathcal{D}}

\newcommand{\calL}{\mathcal{L}}
\newcommand{\calM}{\mathcal{M}}

\newcommand{\calO}{\mathcal{O}}

\newcommand{\bs}{\backslash}

\newcommand{\tr}{\operatorname{tr}}

\newcommand{\Sl}{\operatorname{SL}}
\newcommand{\Gl}{\operatorname{GL}}

\newcommand{\Sp}{\operatorname{Sp}}

\newcommand{\Orth}{\operatorname{O}}

\newcommand{\Hom}{\operatorname{Hom}}

\newcommand{\Sym}{\operatorname{Sym}}

\newcommand{\GL}{\operatorname{GL}}

\newcommand{\CH}{\operatorname{CH}}

\newcommand{\ord}{\operatorname{ord}}

\begin{document}

\title[Vector valued formal Fourier-Jacobi series]{Vector valued formal Fourier-Jacobi series}

\author[Jan H.~Bruinier]{Jan
Hendrik Bruinier}
\address{Fachbereich Mathematik,
Technische Universit\"at Darmstadt, Schlossgartenstrasse 7, D--64289
Darmstadt, Germany}
\email{bruinier@mathematik.tu-darmstadt.de}

\thanks{The author is partially supported by DFG grants BR-2163/2-2 and FOR 1920.}

\date{\today}

\begin{abstract}
H.~Aoki showed that any symmetric formal Fourier-Jacobi series for the symplectic group $\Sp_2(\Z)$ is the Fourier-Jacobi expansion of a holomorphic Siegel modular form. We prove an analogous result for vector valued symmetric formal Fourier-Jacobi series, by combining Aoki's theorem with facts about vector valued modular forms.
Recently, this result was also proved independently by M.~Raum using a different approach.
As an application, by means of work of W.~Zhang, modularity results for special cycles of codimension $2$ on Shimura varieties associated to orthogonal groups can be derived.
\end{abstract}

\maketitle

\section{Introduction}
\label{sect:intro}

Let $\Gamma^{(g)}$ be the metaplectic extension of the integral symplectic group $\Sp_g(\Z)$ of genus $g$.
The elements of $\Gamma^{(g)}$ are pairs $(M,\alpha)$, where  $M=\kzxz{A}{B}{C}{D}\in \Sp_g(\Z)$, and $\alpha$ is a holomorphic function on the Siegel upper half plane $\H_g$ such that $\alpha(Z)^2= \det(CZ+D)$. The product of two elements $(M_1,\alpha_1)$ and $(M_2,\alpha_2)$ is defined by
\[
(M_1,\alpha_1(Z))\cdot (M_2,\alpha_2(Z)) = (M_1M_2,\alpha_1(M_2Z)\alpha_2(Z)).
\]

Let $k\in \frac{1}{2}\Z$, and let
\begin{align}
\label{eq:rho}
\rho:\Gamma^{(g)} \longrightarrow \Gl(V_\rho)
\end{align}
be a unitary representation on a finite dimensional complex vector space $V_\rho$.  Throughout we assume that $\rho$ is trivial on some congruence subgroup of sufficiently large level.
We denote by $M_k^{(g)}(\rho)$ the vector space of holomorphic Siegel modular forms for the group $\Gamma^{(g)}$ of weight $k$ with representation $\rho$, that is, the space of holomorphic functions
$f:\H_g\to V_\rho$ on $\H_g$
satisfying the transformation law
\begin{align}
\label{siegeltrafo}
f(M Z)= \alpha(Z)^{2k}\rho(M,\alpha)f(Z)
\end{align}
for all $(M,\alpha)
\in \Gamma^{(g)}$ (and which are holomorphic at the cusp $\infty$ if $g=1$), see \cite{Fr}.
For the trivial representation $\rho_0$ on $\C$, we briefly write $M_k^{(g)}$ instead of $M_k^{(g)}(\rho_0)$.
Note that the transformation law for $(1,-1)\in \Gamma^{(g)}$ implies that $M_k^{(g)}=0$ if $k$ is not integral.

In the present note we are mainly interested in the case $g=2$.
If we write the variable $Z\in \H_2$ as
\begin{align}
Z=\zxz{\tau}{z}{z}{\tau'}
\end{align}
with $\tau,\tau'\in \H_1$ and $z\in \C$, then any $f\in M_k^{(2)}(\rho)$ has a Fourier-Jacobi expansion of the form
\begin{align}
\label{fjexp}
f(Z)= \sum_{m\geq 0} \phi_m(\tau,z) q'{}^m,
\end{align}
where $q'=e^{2\pi i \tau'}$, and the
coefficients $\phi_m$ are Jacobi forms of weight $k$, index $m$, with
representation $\rho$. To explain this more precisely, recall that the
(metaplectic) Jacobi group $\Gamma^J=\Gamma^{(1)}\ltimes \Z^2$ acts on
the Jacobi half plane $\H_1\times \C$ in the usual way, see
\cite{EZ}. Moreover, there is an embedding $\Sl_2(\Z)\ltimes \Z^2\to \Sp_2(\Z)$,
\begin{align}
\label{jacemb}
\left(\abcd,(\lambda,\mu)\right)\longmapsto \begin{pmatrix}
a&0&b& a\mu-b\lambda \\
\lambda & 1 & \mu & \lambda\mu\\
c & 0 & d & c\mu -d\lambda \\
0 & 0 & 0 & 1
\end{pmatrix},
\end{align}
which is compatible with the group actions on $\H_1\times\C$ and on
$\H_2$.  It lifts to an embedding $j:\Gamma^J\to \Gamma^{(2)}$, taking
the cocycle $\alpha$ to itself.  The restriction of $\rho$ defines a
representation of the metaplectic Jacobi group.  A holomorphic
function $\phi:\H_1\times \C\to V_\rho$ is called a Jacobi form for
$\Gamma^{J}$ of weight $k$ and index $m$ with representation $\rho$, if
the function
\begin{align}
\label{jactrafo}
\tilde \phi(Z) = \phi(\tau,z)q'{}^m
\end{align}
on $\H_2$ satisfies the transformation law \eqref{siegeltrafo} for all
$(M,\alpha)\in j(\Gamma^J)$, and $\phi$ is holomorphic at $\infty$.  Here the latter
condition means that the Fourier expansion of $\phi$ has the form
\begin{align}
\label{fourierphi}
\phi(\tau,z)=\sum_{\substack{n\in \Q\\ n\geq 0}}\sum_{\substack{r\in \Q\\ r^2\leq 4mn}} c(\phi; n,r)q^n \zeta^r
\end{align}
with $q=e^{2\pi i \tau}$, $\zeta=e^{2\pi i z}$, and $c(\phi; n,r)\in
V_\rho$. Because of our assumption that $\rho$ be trivial on some
congruence subgroup, the coefficients $c(\phi; n,r)$ are supported on
rational numbers with bounded denominators. We write $J_{k,m}(\rho)$
for the vector space of Jacobi forms for $\Gamma^{J}$ of weight $k$,
index $m$, with representation $\rho$.

The element
\begin{align}
\delta= \left( \begin{pmatrix}
0& 1& &\\
1& 0& &\\
& & 0 & 1\\
& & 1 & 0
\end{pmatrix}, i\right)\in \Gamma^{(2)}
\end{align}
acts on $\H_2$ by $Z=\kzxz{\tau}{z}{z}{\tau'}\mapsto\kzxz{\tau'}{z}{z}{\tau} $.
The transformation law \eqref{siegeltrafo} implies that the Fourier-Jacobi coefficients \eqref{fjexp}
of any $f\in M_k^{(2)}(\rho)$  satisfy the symmetry relation
\begin{align}
\label{fjsymmetry}
c(\phi_m;n,r)=i^{2k}\rho\left(\delta \right) c(\phi_n;m,r).
\end{align}
This motivates the following definition.

\begin{definition}
A {\em formal Fourier-Jacobi series} for the group $\Gamma^{(2)}$ of weight $k$ with representation $\rho$ is a formal series
\begin{align}
f=\sum_{m\geq 0} \phi_m(\tau,z) q'{}^m
\end{align}
whose coefficients $\phi_m(\tau,z)$ belong to $J_{k,m}(\rho)$.
We denote the vector space of such formal Fourier-Jacobi series by $N^{(2)}_k(\rho)$.
We call $f\in N^{(2)}_k(\rho)$ {\em symmetric} if its coefficients satisfy
\eqref{fjsymmetry}  for all triples $m,n,r$.
\end{definition}

For the trivial representation $\rho_0$ on $\C$, we briefly write $N_k^{(2)}$ instead of $N_k^{(2)}(\rho_0)$.
The above considerations show that
the Fourier-Jacobi expansion of any $f\in M^{(2)}_k(\rho)$ is a symmetric Fourier-Jacobi series. In the present note we consider the following converse.

\begin{theorem}
\label{thm:main}
Every symmetric $f\in N_k^{(2)}(\rho)$
is the Fourier-Jacobi expansion of some Siegel modular form in
$M^{(2)}_k(\rho)$. In particular, $f$ converges absolutely.
\end{theorem}

Such a result was first proved by Aoki in \cite{Ao}
for the {\em trivial\/} representation $\rho$. His proof relies on classical facts on Taylor expansions of Jacobi forms \cite{EZ} and comparisons of dimension formulas.
Variants for paramodular groups of small level and the trivial representation are proved in \cite[Theorem 1.2]{IPY}.
In a recent preprint \cite{Ra}, Raum proves Theorem \ref{thm:main} (for possibly non-trivial representations as in \eqref{eq:rho}), by employing the Hirzebruch-Riemann-Roch theorem and the Lefschetz fixed point formula, combined with asymptotic estimates for the dimensions of symmetric formal Fourier-Jacobi series.

In the present note we show that the full statement of Theorem \ref{thm:main}
can be derived from Aoki's original result for the scalar 
case by employing some
facts about vector valued Siegel modular forms and their Fourier-Jacobi expansions. We hope that this approach might also help
to attack the analogous problem in other situations (e.g.~in higher genus) by reducing the vector valued to the scalar case.

An important application of Theorem \ref{thm:main} is an analogue of the Gross-Kohnen-Zagier theorem \cite{GKZ} for codimension $2$ special cycles  on Shimura varieties associated with orthogonal groups of signature $(n,2)$. It states that the generating series of special cycles of codimension $2$ is a vector valued Siegel modular form for the group $\Gamma^{(2)}$ of weight $1+n/2$ with values in the second  Chow group. This result was conjectured (in greater generality) by Kudla in \cite{Ku:MSRI}, motivated by his joint work with Millson on geometric generating series, see e.g.~\cite{KM3}, \cite{Ku:Duke}.
Employing \cite{Bo3}, it was proved by Zhang \cite{Zh} that the generating series is a symmetric formal Fourier-Jacobi series.
Theorem \ref{thm:main} implies that it actually converges and therefore defines a Siegel modular form. We explain this application in Section \ref{sect:3}.

I thank Martin Raum and the referee for very helpful comments on this paper.

\section{Vector valued modular forms and the proof of Theorem \ref{thm:main}}
\label{sect:2}

Let $(\rho,V_\rho)$ and $(\sigma,V_\sigma)$ be finite dimensional representations of $\Gamma^{(2)}$.
Denote the canonical
bilinear pairing
$\Hom(V_\rho,V_\sigma)\times V_\rho\to V_\sigma$
by $\langle \lambda,v\rangle=\lambda(v)$ for $\lambda\in \Hom(V_\rho,V_\sigma)$ and $v\in V_\rho$.
The following two lemmas are easily seen.

\begin{lemma}
\label{lem:1}
Let  $f=\sum_{m\geq 0} \phi_m q'{}^m\in N_k^{(2)}(\rho)$ and
$g=\sum_{m\geq 0} \psi_m q'{}^m\in N_{k_1}^{(2)}(\sigma)$.
Then
\begin{align*}
f\otimes g&=\sum_{M\geq 0} \sum_{\substack{m,m_1\geq 0\\ m+m_1=M}}\phi_m \otimes \psi_{m_1} q'{}^M
\end{align*}
belongs to $N_{k+k_1}^{(2)}(\rho\otimes\sigma)$.
If $f$ and $g$ are symmetric, then $f\otimes g$ is symmetric as well.
\end{lemma}

\begin{lemma}
\label{lem:2}
Let  $f=\sum_{m\geq 0} \phi_m q'{}^m\in N_k^{(2)}(\rho)$, and let
$g=\sum_{m\geq 0} \psi_m q'{}^m$
be a formal Fourier-Jacobi series of weight $k_1$ and representation
$\Hom(V_\rho,V_\sigma)$.
Then
\begin{align*}
\langle g,f\rangle &=\sum_{M\geq 0} \sum_{\substack{m,m_1\geq 0\\ m+m_1=M}}\langle\psi_{m_1},\phi_m\rangle q'{}^M
\end{align*}
belongs to $N_{k+k_1}^{(2)}(\sigma)$.
If $f$ and $g$ are symmetric, then $\langle f,g\rangle $ is symmetric as well.
\end{lemma}

For the trivial representation, the direct sum over $k\in \Z_{\geq 0}$ of all spaces $N_k^{(2)}$
is a graded algebra with the multiplication of Lemma \ref{lem:1}. Any $f\in N_k^{(2)}$
has a multiplicative inverse in the field $\calM(\H_1\times \C)((q'))$ of
formal Laurent series with coefficients in the field $\calM(\H_1\times \C)$ of meromorphic functions on $\H_1\times\C$.
It is easily seen that the $m$-th coefficient of the inverse is a meromorphic Jacobi form of weight $-k$  and index $m$.

\begin{definition}
\label{def:merfj}
Let $f$ be a meromorphic modular form for $\Gamma^{(2)}$ of weight $k$ with representation $\rho$. Write $f=g/h$ with $g\in M_{k+k_1}^{(2)}(\rho)$ and $h\in M_{k_1}^{(2)}$. We define the {\em formal} Fourier-Jacobi expansion of $f$ as the product of the Fourier-Jacobi expansion of $g$ with the inverse of the Fourier-Jacobi expansion of $h$ in $\calM(\H_1\times\C)((q'))$. Its $m$-th coefficient is a meromorphic Jacobi form of weight $k$ for $\Gamma^{(2)}$ with representation $\rho$.
\end{definition}

Our assumption that $\rho$ is trivial on some congruence subgroup implies that it is always possible to find $g$ and $h$ as required. The definition is independent of their choice. The map taking a meromorphic modular form to its formal Fourier-Jacobi expansion is injective. Note that the question where the formal Fourier-Jacobi expansion actually converges is subtle. We do not address it here. Moreover, note that our approach to
(formal) Fourier-Jacobi expansions of meromorphic Siegel modular forms is different from the one pursued in \cite{DMZ} by means of Fourier integrals along suitably chosen contours.

The center $C=\{(\pm 1,\pm 1)\}\subset \Gamma^{(2)}$ acts trivially on $\H_2$. We consider the $\Gamma^{(2)}$-invariant subspace
\begin{align*}
V_\rho(k) = \{v\in V_\rho:\; \text{$\rho(-1,1)(v)=v=(-1)^{2k}\rho(1,-1)(v)$} \}\subset V_\rho.
\end{align*}
The transformation law \eqref{siegeltrafo} for $j((-1,i),0)$ and the symmetry relation \eqref{fjsymmetry} imply that any $f\in N_k^{(2)}(\rho)$ actually takes values in $V_\rho(k)$.

\begin{proposition}
\label{prop:global}
There exists a positive  $k_0\in k+\Z$ with the following property:
For every $a\in \H_2$ which is not a fixed point of the action of $\Gamma^{(2)}/C$, there are
modular forms $g_1,\dots,g_d\in M^{(2)}_{k_0}(\rho)$
whose values $g_1(a),\dots, g_d(a)$ generate the space $V_\rho(k)$.
 \end{proposition}

\begin{proof}
Let $X$ be the Satake compactification of $\Gamma^{(2)}\bs \H_2$. For $r\in \frac{1}{2}\Z$ we let $\calM_r$ be the sheaf of scalar valued modular forms of weight $r$  on $X$, and
let $\calM_r(\rho)$ be the sheaf of modular forms of weight $r$ with representation $\rho$ on $X$. These sheaves are coherent $\calO_X$-modules and
$\calM_{r+k}(\rho)=\calM_r\otimes \calM_k(\rho)$. By the work of Satake and Baily-Borel, there exists a positive integer $r_0$ such that $\calM_r$ is a very ample line bundle for all positive integers $r$ which are divisible by $r_0$.
Hence, according to a theorem of Serre (see e.g.~Theorem~5.17 in \cite[Chapter~II]{Ha}), the sheaf $\calM_{r+k}(\rho)$ is generated by global sections for all integers $r$ which are divisible by $r_0$ and sufficiently large. We fix such an $r$ and put $k_0=r+k$.

If $a$ is not an elliptic fixed point, then the stalk $\calM_{k_0,a}(\rho)$ at $a$ is equal to $V_\rho(k)\otimes_\C \calO_{X,a}$. By the assumption on $k_0$ there exist global sections of $\calM_{k_0}(\rho)$ which generate the stalk $\calM_{k_0,a}(\rho)$. Their values at $a$ must generate $V_\rho(k)$.
\end{proof}

E.~Freitag has pointed out that this proposition could also be proved using Poincar\'e series as in Theorem~4.4 of \cite[Chapter~I]{Fr}.

\begin{proof}[Proof of Theorem \ref{thm:main}]
Let $f=\sum_m \phi_m q'{}^m\in N_k^{(2)}(\rho)$.
By replacing $\rho$ by its restriction to $V_\rho(k)$, we may assume without loss of generality that $V_\rho=V_\rho(k)$.

Put $d=\dim(V_\rho)$, and
let $k_0\in k+\Z$ be positive number which has the property of Proposition~\ref{prop:global} for the dual representation $\rho^\vee$.
Let $a\in \H_2$ be a point which is not a fixed point of the action of $\Gamma^{(2)}/C$.
We choose $g_1,\dots,g_d\in M^{(2)}_{k_0}(\rho^\vee)$
such that the vectors $g_1(a),\dots,g_d(a)$ form a basis of $V_\rho^\vee$.
The $d$-tuple
\[
g={}^t(g_1,\dots,g_d)
\]
defines a vector valued Siegel modular form with values in $\Hom(V_\rho,\C^d)$ transforming with the representation induced by $\rho$ and the
trivial representation $\rho_0^d$ on $\C^d$.
The pairing
\[
h=\langle g,f\rangle
\]
as in Lemma~\ref{lem:2}
defines an element on $N_{k+k_0}^{(2)}(\rho_0^d)$,
that is, a $d$-tuple of scalar formal Fourier-Jacobi series of weight $k+k_0$.
Since $f$ and $g$ are symmetric, $h$ is also symmetric by Lemma~\ref{lem:2}. According to \cite{Ao} (Theorem \ref{thm:main} for the trivial representation),
$h$ is the Fourier-Jacobi expansion of some Siegel modular form in $M_{k+k_0}^{(2)}(\rho_0^d)$, which we also denote by $h$.

For $Z\in\H_2$ the value $g(Z)\in \Hom(V_\rho,\C_d)$ is an invertible linear map if and only if $\det(g(Z))\neq 0$.
Since $g_1(a),\dots,g_d(a)$ are linearly independent, $g(Z)$ is invertible in a neighborhood of $a$.
Consequently, the assignment
$Z\mapsto g(Z)^{-1}$
defines a {\em meromorphic} modular form $g^{-1}$ for $\Gamma^{(2)}$ of weight $-k_0$ with values in  $\Hom(\C^d,V_\rho)$.

The natural pairing
\[
\langle g^{-1},h\rangle = \langle g^{-1},\langle g,f\rangle \rangle
\]
is a meromorphic Siegel modular form for $\Gamma^{(2)}$ of weight $k$ with representation $\rho$,
whose formal Fourier-Jacobi expansion in the sense of Definition \ref{def:merfj} has to agree with $f$.
It is holomorphic in a neighborhood of $a$.
Varying the point $a$ and the corresponding modular forms $g_\nu$, we find that $f$ is holomorphic on the complement of the set of elliptic fixed points.

By a standard argument, we get a holomorphic continuation of $f$ over the fixed point
manifolds\footnote{By \cite[Theorem 6]{St}, for general $g\geq 2$, the minimal codimension of fixed point manifolds of the action of $\Sp_g(\Z)$ on $\H_g$ is $g-1$.}
of codimension $>1$ of elements of $\Sp_2(\Z)$.
To get a continuation over the fixed point manifolds of codimension $1$, we note that by \cite[\S 7.3]{St} the group
$\Sp_2(\Z)$ has up to conjugation only the two elements
$\kzxz{H}{0}{0}{H}$ and $\kzxz{H}{J}{0}{H}$ whose fixed point manifolds have codimension $1$.
Here $H=\kzxz{1}{0}{0}{-1}$ and $J=\kzxz{0}{-1}{1}{0}$. The corresponding fixed point manifolds are given by the divisors $D_1=\{z=0\}$ and $D_2=\{z=1/2\}$ on $\H_2$.

Now there exists a holomorphic scalar valued Siegel modular form $u$ with divisor supported on the $\Sp_2(\Z)$-translates of $D_1+D_2$ and a holomorphic $V_\rho$-valued modular form $v$ such that $u\cdot f = v$.
All Fourier-Jacobi coefficients of $u$ must vanish along the divisor $\{z=0\}$ on $\H_1\times\C$ with order at least $\ord_{D_1}(u)$. Since the Fourier-Jacobi coefficients of $f$ are holomorphic, we find that all Fourier-Jacobi coefficients of $v$ have to vanish along $\{z=0\}$ with order at least $\ord_{D_1}(u)$. But this implies that $f$ is holomorphic along $D_1$. Analogously it is holomorphic along $D_2$.

Consequently, the formal Fourier-Jacobi series $f$ is the Fourier-Jacobi expansion of an
element of $M_k^{(2)}(\rho)$.
\end{proof}

\begin{remark}
For simplicity we stated Theorem \ref{thm:main} and its proof for vector valued Fourier-Jacobi series with scalar $K$-type $\det^{k}$ at $\infty$.
The above proof can be generalized to vector valued $K$-type $\det^{k}\otimes \Sym^j$ at $\infty$ in a straightforward way.
\end{remark}

\section{Generating series of special cycles}
\label{sect:3}

Here we describe the application of Theorem \ref{thm:main} to Kudla's
modularity conjecture, see \cite[Section 3, Problem 1]{Ku:MSRI}.

Let $(V,Q)$ be a quadratic space over $\Q$ of signature $(n,2)$. The
hermitian symmetric space corresponding to the orthogonal group of $V$
can be realized as a connected component $\calD^+$ of the complex
manifold
\[
\calD=\{[z]\in P(V(\C)):\; \text{$(z,z)=0$ and $(z,\bar z)<0$}\}.
\]
Here $P(V(\C))$ denotes the projective space of
$V(\C)=V\otimes_\Q\C$. Let $L\subset V$ be an even lattice, and write $L'$
for its dual.  Let $\Orth(L)$ be the orthogonal group of $L$, and let
$\Gamma_L\subset \Orth(L)$ be a subgroup of
finite index which acts trivially on $L'/L$ and which takes $\calD^+$
to itself. By the theory of Baily-Borel the quotient $X(\Gamma_L)= \Gamma_L\bs \calD^+$ has a structure as
a quasi-projective algebraic variety.
The
tautological line bundle $\calL$ over $\calD$ descends to a line
bundle over $X(\Gamma_L)$, the line bundle of modular forms of weight
$1$.

Special cycles on $X(\Gamma_L)$ can be defined as follows, see \cite{Ku:Duke}, \cite{Ku:MSRI}.
For $1\leq r\leq n$, let $S _{L,r}$ be the complex vector space of
functions $\varphi:(L'/L)^r\to \C$. Recall that there is a Weil
representation
\[
\omega_{L,r}:\Gamma^{(r)}\longrightarrow \GL(S_{L,r}).
\]
If $n$ is even, it is a subrepresentation of the restriction to
$\Sp_r(\Z)$ of the usual Weil representation of $\Sp_r(\hat\Q)$ on the
space of Schwartz-Bruhat functions on $V(\hat\Q)^r$. If $n$ is odd, in
addition the cocycles have to be matched, see e.g.~\cite{Ku:Integrals}
for the case $r=1$, the case of general $r$ is analogous.

For an $r$-tuple $\lambda=(\lambda_1,\dots,\lambda_r)\in V^r$  we let
$Q(\lambda)=\frac{1}{2}((\lambda_i,\lambda_j))_{i,j}\in \Q^{r\times r}$ be the corresponding matrix of inner products.
If $Q(\lambda)$ is positive semidefinite of rank $r(\lambda)\in \{0,\dots,r\}$, then
\[
Z(\lambda)=\{[z]\in \calD^+: \; (z,\lambda_1)=\ldots=(z,\lambda_r)=0\}
\]
is a submanifold of codimension $r(\lambda)$. If $Q(\lambda)$ is not positive semidefinite, then $\calD^+(\lambda)=\emptyset$. For every positive semidefinite symmetric matrix $T\in\Q^{r\times r} $ of rank $r(T)$, and for $ \varphi \in S_{L,r}$, we define a cycle
\[
Z(T,\varphi)=\sum_{\substack{\lambda\in L'{}^r\\Q(\lambda)=T}}\varphi(\lambda) \cdot Z(\lambda)
\]
of codimension $r(T)$ with complex coefficients. By reduction theory, it descends to an algebraic
cycle on the quotient $X(\Gamma_L)$, which we also denote by
$Z(T,\varphi)$.  If the denominators of the entries of $T$ do not
divide the level of $L$, then $Z(T)=\emptyset$.  By taking the
intersection pairing $Z(T,\varphi)\cdot (\calL^\vee)^{r-r(T)}$ in the sense of \cite[Chapter 2.5]{Fu} with a power of the dual bundle of $\calL$, we
obtain a cycle class of codimension $r$, that is, an element of
$\CH^r(X(\Gamma_L))_\C$.
We write
$Z(T)$ for the element
\[
\varphi\mapsto  Z(T,\varphi)\cdot (\calL^\vee)^{r-r(T)}
\]
of $\Hom(S_{L,r}, \CH^{r}(X(\Gamma_L))_\C)$.

\begin{conjecture}[Kudla]
\label{conj:kudla}
The formal generating series
\[
A_r(Z)=\sum_{\substack{T\in \Q^{r\times r}\\ T\geq 0}} Z(T) \cdot  q^T,
\]
valued in $S_{L,r}^\vee\otimes_\C \CH^{r}(X(\Gamma_L))_\C$, is a Siegel modular form of genus $r$ in $M_{1+n/2}^{(r)}(\omega_{L,r}^\vee)$ with values in $\CH^{r}(X(\Gamma_L))_\C$. Here we have put $q^T=e^{2\pi i \tr(T Z)}$ for $Z\in \H_r$.
\end{conjecture}

For $r=1$ this conjecture was proved by Borcherds in \cite{Bo3}.

\begin{theorem}
\label{thm:kudla}
For $r=2$ the conjecture is true.
\end{theorem}

 \begin{proof}
 Zhang showed in \cite{Zh} that $A_2(Z)$ is a formal Fourier-Jacobi series  in $N_{1+n/2}^{(2)}(\omega_{L,2}^\vee)$ with values in $\CH^{2}(X(\Gamma_L))_\C$.
 Theorem \ref{thm:main} therefore implies the assertion.
 \end{proof}


\end{document}